\documentclass{article}

\usepackage{hyperref}
\usepackage[bibtex-style]{amsrefs}
\usepackage{amsmath, amsthm, amssymb, amsfonts}
\usepackage[utf8]{inputenc}
\usepackage[T1]{fontenc}
\usepackage{bm}
\usepackage{tikz}
\usetikzlibrary{calc}
\usetikzlibrary{decorations}
\usetikzlibrary{decorations.pathreplacing}
\usetikzlibrary{shapes}
\usetikzlibrary{arrows, automata, decorations.pathreplacing, fit, matrix, patterns, positioning}
\usepackage{tikz-qtree}
\usepackage{xifthen} 
\usepackage{xspace}
\usepackage{lineno}
\usepackage{bbold}
\usepackage[mathscr]{euscript}
\usepackage{enumitem}

\usepackage{color}
\definecolor{lightgray}{rgb}{0.8, 0.8, 0.8}
\definecolor{darkgray}{rgb}{0.7, 0.7, 0.7}
\definecolor{darkblue}{rgb}{0, 0, .4}

\newcommand\absdot[2]{
	\node at #1 {\normalsize $\bullet$};
	\node at #1 [below] {$#2$};
}

\newcommand{\plotperm}[1]{
	\foreach \j [count=\i] in {#1} {
		\absdot{(\i,\j)}{};
	};
}

\newcommand{\plotpermbox}[4]{
	\draw [darkgray, ultra thick, line cap=round]
		({#1-0.5}, {#2-0.5}) rectangle ({#3+0.5}, {#4+0.5});
}

\newcommand{\plotpermgraph}[1]{
	\foreach \j [count=\i] in {#1} {
		\foreach \b [count=\a] in {#1} {
			\ifthenelse{\a<\i \AND \b>\j}{\draw (\a,\b)--(\i,\j);}{}
		};
	};
	\plotperm{#1};
}

\newcommand{\Av}{\mathop{\mathrm{Av}}}
\newcommand{\Ba}{\mathop{\mathrm{Ba}}}
\newcommand{\Avs}{\mathop{\mathrm{Av}_{\mathsf{S}}}}
\newcommand{\Obs}{\mathop{\mathrm{Ba}_{\mathsf{S}}}}

\newcommand{\C}{\ensuremath{{\mathcal C}}\xspace}
\newcommand{\cD}{\ensuremath{{\mathcal D}}\xspace}
\renewcommand{\S}{\ensuremath{{\mathcal S}}\xspace}
\newcommand{\E}{\ensuremath{{\mathcal E}}\xspace}
\newcommand{\X}{\ensuremath{{\mathcal X}}\xspace}
\newcommand{\Y}{\ensuremath{{\mathcal Y}}\xspace}
\newcommand{\Z}{\ensuremath{{\mathcal Z}}\xspace}

\newcommand{\Sep}{\ensuremath{\mathsf{Sep}}\xspace}
\newcommand{\I}{\ensuremath{\mathsf{I}}\xspace}
\newcommand{\D}{\ensuremath{\mathsf{D}}\xspace}
\newcommand{\La}{\ensuremath{\mathsf{L}}\xspace}

\newcommand{\cA}{\ensuremath{\mathscr{A}}\xspace}
\newcommand{\uni}{\ensuremath{\mathbb{1}}\xspace}
\newcommand{\cUs}{\ensuremath{\mathscr{U}_\mathsf{S}}\xspace}
\newcommand{\Rep}{\ensuremath{\mathscr{R}}\xspace}
\newcommand{\cl}{\mathop{\mathrm{cl}}}
\newcommand{\rel}{\ensuremath{\sim}\xspace}

\newcommand{\Fam}[1]{\ensuremath{\mathfrak{#1}}}

\newcommand{\FSplit}{\Fam{Split}}

\newcommand{\merge}{\ensuremath{\odot}\xspace}

\newcommand{\avp}[1]{%
\raisebox{-2pt}{\begin{tikzpicture}
\path[use as bounding box] (-0.18,-0.18) rectangle (0.1,0.1);
\begin{scope}[scale=0.04]
\fill[red]  (0,0) circle (4.56);
\fill[white] (0,0) circle (4.0);
\fill[fill=red] (-2.8,3.2) -- (3.2,-2.8) -- (2.8,-3.2) -- (-3.2,2.8) -- cycle;
\end{scope}
\node at (0,0) {\textup{\textsf{\tiny #1}}};
\end{tikzpicture}
}}

\newcommand{\acb}{\ensuremath{\avp{132}}\xspace}
\newcommand{\bac}{\ensuremath{\avp{213}}\xspace}
\newcommand{\bca}{\ensuremath{\avp{231}}\xspace}
\newcommand{\cab}{\ensuremath{\avp{312}}\xspace}

\newtheorem{theorem}{Theorem}
\newtheorem{corollary}[theorem]{Corollary}
\newtheorem{lemma}[theorem]{Lemma}

\newtheorem{observation}[theorem]{Observation}

\newtheorem*{question}{Question}

\title{Unsplittable classes of separable permutations}
\author{Michael Albert  \and Vít Jelínek\thanks{Supported by project 16-01602Y of 
the Czech Science Foundation. This work also received financial support from the Neuron Foundation for Support of Science.}}
\date{}

\begin{document}
\maketitle

\begin{abstract}
A permutation class is splittable if it is contained in the merge of two of its proper subclasses. We characterise the unsplittable subclasses of the class of separable permutations both structurally and in terms of their bases.
\end{abstract}

\section{Introduction}

In recent years one of the main areas of investigation within the study of pattern-avoiding permutations, or permutation classes has been to develop structural characterisations for some classes in terms of simpler ones. Such a characterisation can have many useful consequences including: exact or approximate enumerations, determination of the type (e.g., rational, algebraic) of the generating function of a class, and algorithms for membership in such classes. The survey articles of Vatter \cite{Vatter2015Permutation-cla} and Brignall \cite{Brignall2010A-survey-of-sim} are perhaps the best introduction to this endeavour.

Among the constructions used to build new classes from old, one of the most natural to consider is that of \textit{merging} permutations (or permutation classes). We delay formal definitions to the next section, but briefly a permutation $\pi$ is a merge of two permutations $\sigma$ and $\tau$ if its elements can be partitioned into two sets which are isomorphic (in the sense of relative ordering of corresponding pairs of points) to $\sigma$ and $\tau$ respectively. However, this construction has been relatively ignored in the literature.  Merges been used in  \cite{Bona2012On-the-Best-Upp, Bona2014A-New-Upper-Bou, Claesson2012Upper-bounds-fo} to  establish bounds on the growth rates of certain permutation classes. A restricted version of the construction was used in \cite{Albert2010Growth-rates-fo} to show that certain collections of classes have equal growth rates. Those classes which are obtained by merging the classes consisting of monotone increasing or monotone decreasing permutations were all shown to be finitely based in \cite{Kezdy1996Partitioning-pe}, and also played a role in \cite{Albert2007On-the-length-o} (as evidence for a conjecture concerning the lengths of longest pattern avoiding subsequences in random permutations). Apparently the first paper where merging, and its dual, splitting, played a central role was \cite{Jelinek2015Splittings-and-}.

Call a class \textit{splittable} if it is contained in the merge of two of its proper subclasses. Then, a natural question to ask is what are the essential building blocks with respect to this notion - that is, can we characterise unsplittable classes? The previously mentioned paper \cite{Jelinek2015Splittings-and-} began the study of this and related notions, and this paper continues it in a somewhat more limited context. Namely, aside from some initial general observations we work within the class $\Sep$ of separable permutations. This class of permutations is the natural analog of the graph class consisting of the \emph{cographs} -- those graphs constructible from the one vertex graph by complementation and disjoint union. As a tool to help characterise the unsplittable subclasses of $\Sep$ we investigate an interesting monoid structure on the so-called representable subclasses of $\Sep$ -- this structure is key to showing that the representable classes and the unsplittable ones coincide. 

As we will discuss in the conclusions, we view this work as representing the first steps of the investigation of splittability within well-behaved classes. 

\section{Terminology, notation and basic observations}

Let $n$ be a non-negative integer, and let $[n] = \{1,2,\dots,n\}$. A \emph{permutation} (of $[n]$) is a bijective map $\pi\colon [n] \to [n]$. The set of all permutations of $[n]$ is denoted $\S_n$ and the set of all permutations is denoted $\S$. If $\pi \in \S_n$ then we say that the \emph{size} of $\pi$ is $n$ and write $|\pi| = n$. We frequently represent permutations in \emph{one line notation} as the sequence of their values $\pi_1 \pi_2 \dotsb \pi_n$ (where $\pi_i = \pi(i)$). Since a permutation is also a set of points in the plane, i.e., the set of points $(i, \pi_i)$ for $i \in [n]$, we can speak of relationships among these points in the horizontal direction thought of as position (using words like before, after, to the left of, or to the right of), or in the vertical direction thought of as value (using words like above, below, greater or lesser). For instance, the set of elements lying above the leftmost point in $\pi$ consists of the subset (or subsequence in the one line version) of all those points $(j, \pi_j)$ with $\pi_1 < \pi_j$.

The geometric viewpoint of permutations as sets of points in the plane makes it clear that there are eight symmetries corresponding to the symmetries of the square that can be applied to permutations. In one line notation, these are generated by reversal, complementation (which, for $\pi \in \S_n$ replaces each value $i$ by $n+1-i$), and ordinary functional inverse. All of our results respect these symmetries which frequently reduces the number of cases that need to be considered.

Given $\pi \in \S_n$ and a $k$-element subset, $\Sigma$, of the permutation $\pi$ there is a unique permutation $\sigma \in \S_k$ such that there exists a correspondence between $\Sigma$ and $\sigma$ that preserves both the left to right, or positional, ordering, and also the bottom to top, or value, ordering. We say that $\sigma$ is the \emph{pattern} of the elements $\Sigma$. For instance, the pattern of the elements of $256143$ occurring in the first, third and sixth positions (i.e., corresponding to the subsequence $263$) is the permutation $132$. If $\sigma \in \S_k$ is the pattern of some subset of $\pi$ then we say that $\sigma$ is \emph{contained in} $\pi$ and write $\sigma \leq \pi$. This relationship is clearly a partial order on permutations. When we wish to draw attention to a particular subset of $\pi$ whose pattern is $\sigma$ we sometimes say that $\sigma$ is a \emph{subpermutation} of $\pi$.

A \emph{permutation class} or \emph{class of permutations}, $\C$, is a set of permutations hereditarily closed downwards with respect to containment. That is, if $\pi \in \C$ and $\sigma \leq \pi$ then $\sigma \in \C$. Note that our definition permits empty permutations, and the empty permutation belongs to every non-empty permutation class. One way to characterise a permutation class is through a set of \emph{forbidden permutations} or \emph{obstructions}. That is, given a set $F$ of permutations we can define the permutation class of permutations \emph{avoiding} $F$:
\[
\Av(F) = \{ \pi \in \S \, \colon \, \forall \, \sigma \in F \: \sigma \not\leq \pi \}.
\]
For any permutation class $\C$ we define the \emph{basis} of $\C$, $\Ba(\C)$ to be the set of minimal elements of $\S \setminus \C$ (with respect to $\leq$). Then  $\C = \Av(\Ba(\C))$ and also if $F$ is an antichain in the ordering on permutations, then $F = \Ba(\Av(F))$.

\begin{figure}
\begin{center}
	$\pi\oplus\sigma=$
	\begin{tikzpicture}[scale=0.2, baseline=(current bounding box.center)]
		\plotpermbox{0}{0}{2}{2};
		\plotpermbox{3}{3}{5}{5};
		\node at (1,1) {$\pi$};
		\node at (4,4) {$\sigma$};
	\end{tikzpicture}
\quad\quad\quad\quad
	$\pi\ominus\sigma=$
	\begin{tikzpicture}[scale=0.2, baseline=(current bounding box.center)]
		\plotpermbox{0}{3}{2}{5};
		\plotpermbox{3}{0}{5}{2};
		\node at (1,4) {$\pi$};
		\node at (4,1) {$\sigma$};
	\end{tikzpicture}
\end{center}
\caption{The sum and skew sum operations.}
\label{fig-sums}
\end{figure}
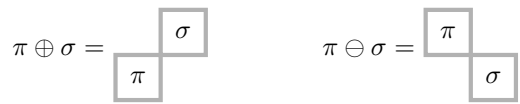

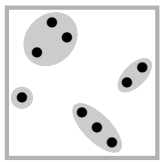
\begin{figure}
\begin{center}
	\begin{tikzpicture}[scale=0.2, baseline=(current bounding box.center)]
		\draw [lightgray, fill] (1,4) circle (20pt);
		\draw[lightgray, fill, rotate around={-45:(2.9,7.9)}] (2.9,7.9) ellipse (45pt and 55pt);
		\draw[lightgray, fill, rotate around={45:(6,2)}] (6,2) ellipse (25pt and 60pt);
		\draw[lightgray, fill, rotate around={-45:(8.5,5.5)}] (8.5,5.5) ellipse (20pt and 40pt);
		\plotpermbox{0.5}{0.5}{9.5}{9.5};
		\plotperm{4,7,9,8,3,2,1,5,6};
	\end{tikzpicture}
\end{center}
\caption{The plot of $479832156$, an inflation of $2413$, specifically $2413[1,132, 321, 12]$.}
\label{fig-479832156}
\end{figure}

There are various operations that apply to permutations and by extension to permutation classes. The \emph{sum}, $\pi \oplus \sigma$ of two permutations $\pi$ and $\sigma$ is the permutation obtained by placing a copy of $\sigma$ above and to the right of a copy of $\pi$. Formally, if $\pi \in \S_n$ and $\sigma \in \S_k$, then $\pi \oplus \sigma \in \S_{n+k}$ and $(\pi \oplus \sigma)(i) = \pi(i)$ for $i \leq n$, while $(\pi \oplus \sigma)(i) = n + \sigma(i-n)$ for $i > n$. The \emph{skew-sum} $\pi \ominus \sigma$ is similarly defined, except $\sigma$ is placed below and to the right of $\pi$. See Figure~\ref{fig-sums}. If a permutation, $\pi$, can be written as a sum of two non-empty permutations then we say that $\pi$ is \emph{sum-decomposable}. If not, then of course we say that $\pi$ is \emph{sum-indecomposable}. \emph{Skew-decomposable} and \emph{skew-indecomposable} are defined similarly. Clearly, sum and skew-sum are associative (but not commutative) operations on permutations, and every permutation can be uniquely represented as a sum of one or more sum-indecomposable permutations or as a skew-sum of skew-indecomposable permutations. A permutation can be both sum- and skew-indecomposable (for instance $2413$) but every sum-decomposable permutation is skew-indecomposable and vice versa.

If $\C$ and $\cD$ are permutation classes then $\C \oplus \cD$ (resp.~$\C \ominus \cD$) is defined to be the set of all permutations which are the sum (skew-sum) of a permutation in $\C$ with one in $\cD$. It is routine to verify that $\C \oplus \cD$ ($\C \ominus \cD$) is also a permutation class.

Sums and skew-sums are a particular case of a more general operation called \emph{inflation}. An \emph{interval} in a permutation is a contiguous set of positions on which the values also form a contiguous set. We say that an interval is \emph{proper} if it is not the whole permutation, and \emph{non-trivial} if it is not a singleton. A permutation whose only intervals are singletons and itself is called \emph{simple}. For instance, in the permutation $2756143$ the second through fourth elements form an interval (positions from 2 to 4, and values from 5 to 7), and the last two elements also form an interval. 

Given $\sigma \in \S_k$ and permutations $\tau_1, \tau_2, \dots, \tau_k$ the \emph{inflation}, $\pi = \sigma[\tau_1, \tau_2, \dots, \tau_k]$ is a permutation of size $\sum_{i=1}^{k} |\tau_i|$ obtained from $\sigma$ by replacing each element $\sigma_i$ by an interval whose pattern is $\tau_i$ where the relationships by position and value between the intervals corresponding to $\sigma_i$ and $\sigma_j$ are the same as those between $\sigma_i$ and $\sigma_j$ themselves. For instance $2756143 = 2413[1,312,1,21]$, and see also Figure \ref{fig-479832156}. Observe that $\sigma \oplus \tau = 12[\sigma, \tau]$ while $\sigma \ominus \tau = 21[\sigma, \tau]$. For permutation classes $\C$ and $\cD$ we define $\C[\cD]$ to be the set of all permutations obtained by inflating a permutation from $\C$ with permutations from $\cD$. Again it is clear that $\C[\cD]$ is a permutation class.

Some particular permutation classes that we will be using repeatedly are:
\begin{itemize}
\setlength{\itemsep}{0pt}
\item
\uni, which contains only the permutation 1 (and the empty permutation - we omit to mention this in what follows),
\item
\I, which contains only the increasing permutations,
\item
\D, which contains only the decreasing permutations,
\item
$\acb = \Av(132)$, together with its symmetries $\bac = \Av(213)$, $\bca = \Av(231)$ and $\cab = \Av(312)$.
\item
$\La$, the class of \emph{layered} permutations, which is the closure of $\D$ under $\oplus$.
\end{itemize}

Note that $\uni=\Av(12,21)$, $\I=\Av(21)$, $\D=\Av(12)$ and $\La=\Av(231,312)=\I[\D]$.

A permutation class, $\C$, is said to be \emph{sum-closed} if for all $\sigma, \tau \in \C$ we have $\sigma \oplus \tau \in \C$. \emph{Skew-closed} is defined similarly.

\begin{observation}
\label{obs-IandDclosure}
A class $\C$ is sum-closed if and only if $\C = \I[\C]$ and skew-closed if and only if $\C = \D[\C]$.
\end{observation}

The following trivial observation can be surprisingly useful in practice:

\begin{observation}
\label{obs-closure-basis}
Let $\C$ be a permutation class and $\cD$ a non-empty subclass of $\C$. Then $\C = \C[\cD]$ if and only if no basis element of $\C$ belongs to $\C[\cD]$, and $\C = \cD[\C]$ if and only if no basis element of $\C$ belongs to $\cD[\C]$.
\end{observation}

Some of the consequences of this observation are:
\begin{itemize}
\setlength{\itemsep}{0pt}
\item
$\C$ is $\oplus$-closed, i.e., $\C = \I[\C]$ if and only if all the basis elements of $\C$ are sum-indecomposable. 
\item
$\C = \C[\C]$ if and only if all the basis elements of $\C$ are simple. 
\item
$\C = \C[\I]$ if and only if no basis element of $\C$ has a non-trivial interval in $\I$ (such permutations are sometimes called \emph{plus-irreducible}).
\end{itemize}

\subsection{Separable permutations}

Permutation classes whose bases consist only of simple permutations are closed under inflation. We have already encountered three such classes: \uni, \I and \D. Any other class of this type must contain both 12 and 21, and if we begin from the class $\{\emptyset, 1, 12, 21\}$ and then form its closure under inflation we obtain the class \Sep of \emph{separable} permutations. \Sep is perhaps most naturally defined recursively as the smallest non-empty class having the property that, if $\sigma, \tau \in \Sep$ then also $\sigma \oplus \tau, \sigma \ominus \tau \in \Sep$. Hence, every non-singleton permutation in \Sep is either sum- or skew-decomposable. The basis of \Sep is $\{2413, 3142\}$ (this is easily verified, but see, e.g., Bose et al.~\cite{BBL}).

To every permutation $\pi \in \Sep$, we associate a \emph{decomposition tree} $T(\pi)$. The leaves of $T(\pi)$ correspond to the elements of $\pi$, and there are two types of internal nodes: $\oplus$-nodes and $\ominus$-nodes. The  (obvious) correspondence between trees and permutations may be defined  recursively: the singleton permutation is represented by a tree consisting of a  single leaf. If $\pi$ is a sum-decomposable permutation of the form $\pi=\pi_1\oplus\pi_2\oplus\dotsb\oplus\pi_k$ with $\pi_i$ being  sum-indecomposable, then $T(\pi)$ has a root which is a $\oplus$-node with $k$  children, whose subtrees are $T(\pi_1),\dotsc,T(\pi_k)$. Skew-decomposable  permutations are handled analogously. Note that in a decomposition tree every child of a $\oplus$-node is a $\ominus$-node or a leaf, and every child of a $\ominus$-node is a $\oplus$-node or a leaf.

The \emph{depth} of a node $w$ in a tree $T$ is the number of edges on the path from $w$ to the root. The depth of $T$ is then the maximum of the depths of the leaves of~$T$. A \emph{bottom leaf} of $T$ is a leaf whose depth is equal to the depth of~$T$. The depth of a separable permutation is the depth of its decomposition tree. The \emph{least common ancestor} of two distinct leaves in  a tree $T$ is the unique internal node of maximum depth which is a common ancestor of the two leaves.


We say that a tree $T$ is \emph{slim} if each of its internal nodes has exactly two children, and at least one of these children is a leaf. Note that in a slim tree $T$ of depth $d$, there is exactly one internal node of depth $i$ for any $i\in\{0,\dotsc,d-1\}$. A \emph{slim} permutation is a separable permutation whose decomposition tree is slim.

\subsection{Merging permutations and splitting classes}

The last construction we need to introduce for permutations, which is central to this paper, is the \emph{merge}. Given two permutations $\sigma$ and $\tau$ of sizes $k$ and $\ell$ respectively the \emph{merges} of $\sigma$ and $\tau$ are those permutations of size $k+l$ whose elements can be partitioned into a set of $k$ elements whose pattern is $\sigma$ and a set of $\ell$ elements whose pattern is $\tau$. For example, all five non-decreasing permutations of size 3 are merges of $12$ with $1$, and all permutations of size 4 except 1234, 2143, 3412, and 4321 are merges of 12 with 21. We write the set of merges of $\sigma$ with $\tau$ as $\sigma \odot \tau$, and use the same notation for permutation classes. Again it is clear that the merge of two permutation classes is itself a permutation class. At the level of permutation classes, $\odot$ is an associative and commutative operation. We frequently think of the partition corresponding to a merge as being given by a red-blue colouring of the elements of a permutation. In that context, to prove $\pi \not \in \C \odot \cD$ requires showing that for every red-blue colouring of $\pi$ either the red elements do not belong to $\C$ (i.e., contain a basis element of $\C$) or the blue elements do not belong to $\cD$. No proper permutation class is closed under the merge operation since by merging $n$ copies of 1 we can produce every permutation in $\S_n$.


We say that a permutation class $\C$ is \emph{splittable}, if there is a sequence 
$(\C_i)_{i=1}^k$ of proper subclasses of $\C$ such that $\C \subseteq \odot_{i=1}^k \C_i$. 
If $X$ is a set of permutation classes, we say that $\C$ is $X$-splittable and write 
$\C \in \FSplit(X)$ if there is a finite sequence $(\C_i)_{i=1}^k$ of classes, each 
belonging to $X$ such that $\C$ is $(\C_i)_{i=1}^k$-splittable. In the case where 
$X$ is a singleton, say $X = \{ \cD \}$ then we often simply say that $\C$ is $\cD$-splittable.

A class is called \emph{atomic} if it is not the union of a finite collection (hence two) of its proper subclasses. Obviously any unsplittable class is atomic since $\X \odot \Y$ always contains $\X \cup \Y$. The main property of atomic classes that we need is contained in the following observation (due in a more general context to Fra{\"\i}ss{\'e} \cite{Fraisse2000Theory-of-relat}).

\begin{lemma}
A class $\C$ is atomic if and only if for all $\pi, \sigma \in \C$ there exists $\tau \in \C$ with $\pi, \sigma \leq \tau$.
\end{lemma}

\begin{proof}
Suppose that there exist $\pi, \sigma \in \C$ having no common extension in $\C$. Let $\cD = \C \cap \Av(\pi)$ and $\E = \C \cap \Av(\sigma)$. Then both $\cD$ and $\E$ are proper subclasses of $\C$ but $\C = \cD \cup \E$ so $\C$ is not atomic. Conversely, if $\C$ is not atomic, say $\C = \cD \cup \E$ where $\cD$ and $\E$ are proper subclasses of $\C$ then we choose $\pi \in \C \setminus \cD$ and $\sigma \in \C \setminus \E$. Any common extension of $\pi$ and $\sigma$ belongs to neither $\cD$ nor $\E$, hence not to $\C$ either.
\end{proof}

The following lemma provides us with a convenient characterisation of splittability:

\begin{lemma}
A class $\C$ is splittable if and only if there exist permutations $\pi, \pi' \in \C$ such that every $\sigma \in \C$ has a red-blue colouring whose red part avoids $\pi$ and whose blue part avoids $\pi'$.
\end{lemma}

\begin{proof}
Suppose that $\C$ is splittable and choose a finite sequence $(\C_i)_{i=1}^k$ of proper subclasses of $\C$ of minimum possible length such that $\C\subseteq C_1\odot\C_2\odot\dotsb\odot\C_k$. Then $k = 2$ since if $k > 2$ we could take $\cD = \C \cap (\C_2 \odot \dotsb \odot \C_k)$ a proper subclass of $\C$ (by the assumption of minimality) but then $\C \subseteq \C_1 \odot \cD$. So $\C \subseteq \cD \odot \E$ for some proper subclasses $\cD$ and $\E$ of $\C$. Now we can simply choose $\pi \in \C \setminus \cD$ and $\pi' \in \C \setminus \E$.

Conversely, suppose that such $\pi$ and $\pi'$ exist and set $\cD = \C \cap \Av(\pi)$ and $\E = \C \cap \Av(\pi')$. Then $\cD$ and $\E$ are proper subclasses of $\C$ and, by the given condition, $\C \subseteq \cD \odot \E$, hence $\C$ is splittable.
\end{proof}

Now we can provide a supply of unsplittable classes.

\begin{lemma}
\label{lem-subs-closed}
If $\C = \C[\C]$, i.e., $\C$ is closed under inflation, then $\C$ is unsplittable.
\end{lemma}

\begin{proof}
Let $\pi$ and $\pi'$ be any two permutations in $\C$ and consider $\tau = \pi[\pi',\dots,\pi']$. Any red-blue colouring of $\tau$ contains either an entirely blue interval of pattern $\pi'$ or a red point in each such interval, and hence its red elements contain $\pi$. So $\C$ cannot be splittable.
\end{proof} 

\begin{lemma}
\label{lem-infl} 
If $\X$ and $\Y$ are unsplittable classes, then so is~$\X[\Y]$.
\end{lemma}

\begin{proof}
Suppose that $\X[\Y]$ is splittable. That means there are permutations $\pi,\pi'\in \X[\Y]$ such that each element of $\X[\Y]$ has a red-blue coloring with no red $\pi$ and no blue~$\pi'$. 

Since $\pi$ is in $\X[\Y]$, it can be expressed as $\pi=\sigma[\tau_1,\dotsc,\tau_k]$ for some $\sigma\in \X$ and $\tau_1,\dotsc,\tau_k\in \Y$. Since $\Y$ is unsplittable, and therefore atomic, there is a permutation $\tau \in \Y$ which contains all the $\tau_1,\dotsc,\tau_k$. We may thus actually assume that $\pi=\sigma[\tau,\tau,\dotsc,\tau]$. Similarly, we have $\pi'=\sigma'[\tau',\tau',\dotsc,\tau']$ for some $\sigma'\in \X $ and $\tau'\in \Y$.

Since $\X$ is unsplittable, it contains a permutation $\sigma^+$ whose every red-blue coloring has a red copy of $\sigma$ or a blue copy of $\sigma'$. Likewise, $\Y$ has a permutation $\tau^+$ containing a red $\tau$ or a blue $\tau'$ in any red-blue coloring. 

Consider the permutation $\pi^+=\sigma^+[\tau^+,\tau^+,\dotsc,\tau^+]\in \X[\Y]$. We claim that any red-blue coloring of $\pi^+$ has a red $\pi$ or a blue~$\pi'$. Fix a red-blue coloring $c$ of~$\pi^+$. Use the coloring $c$ to define a red-blue coloring $d$ of $\sigma^+$ as follows: an element $\sigma^+_i$ is red in $d$ if the $\tau^+$-copy in $\pi^+$ obtained by inflating $\sigma^+_i$ has a red copy of $\tau$ in the coloring $c$; similarly, $\sigma^+_i$ is blue in $d$ if the corresponding $\tau^+$-copy has a blue copy of~$\tau'$. The coloring $d$ of $\sigma^+$ has either a red copy of $\sigma$ or a blue copy of $\sigma'$. In the first case, we find a red copy of $\sigma[\tau,\dotsc,\tau]=\pi$ in $\pi^+$, while in the other case we find a blue~$\pi'$.
\end{proof}

\begin{observation}\label{obs-inf}
Let $\X$ be a permutation class, and let $\Y$ be a class which is $\{\Y_1,\dotsc,\Y_k\}$-splittable for some $\Y_1,\dotsc,\Y_k$. Then $\X[\Y]$ is $\{\X[\Y_1],\dotsc,\X[\Y_k]\}$-splittable, and $\Y[\X]$ is $\{\Y_1[\X],\dotsc,\Y_k[\X]\}$-splittable.
\end{observation}

We need two results from \cite{Jelinek2015Splittings-and-}, whose proofs we omit. 

\begin{lemma}[{\cite[Proposition 3.2]{Jelinek2015Splittings-and-}}]  
\label{lem-abc} 
For any three permutations $\alpha,\beta,\gamma$, 
\[
\Av(\alpha \oplus \beta \oplus \gamma) \subseteq \Av(\alpha \oplus \beta) \odot \Av(\beta \oplus \gamma).
\]
\end{lemma}

\begin{lemma}[{\cite[Theorem 3.15]{Jelinek2015Splittings-and-}}]
\label{lem-1p} 
Let $\pi$ be a  sum-indecomposable permutation. If $\Av(\pi)$ is  $\{\Av(\pi_1), \Av(\pi_2), \dotsc, \Av(\pi_k)\}$-splittable for a set $\{\pi_1,\dotsc,\pi_k\}$ of sum-indecomposable permutations, then  $\Av(1\oplus\pi)$ is $\{\Av(1\oplus\pi_1), \Av(1\oplus\pi_2), \dotsc, \Av(1\oplus\pi_k)\}$-splittable.
\end{lemma}

These results allow us to show that, when avoiding separable permutations, in order to obtain an unsplittable class all the basis elements must be slim.

\begin{lemma}
\label{lem-slim}
Let $\pi$ be a separable sum-decomposable permutation of depth~$d$. Then $\Av(\pi)$ is splittable over the set of classes $\Av(\sigma)$ where $\sigma$ ranges over the slim sum-decomposable subpermutations of $\pi$. \end{lemma}

\begin{proof}
Proceed by induction on $d$. The case $d=0$ is trivial, as the only permutation of depth 0 is the singleton permutation. Suppose $\pi$ has depth  $d \ge 1$. We may write $\pi$ as
\[
\pi=\sigma_1\oplus\sigma_2\oplus\dotsb\oplus\sigma_m
\]
where the $\sigma_i$ are sum-indecomposable permutations of depth at most~$d-1$.  Applying Lemma~\ref{lem-abc} repeatedly, we get
\begin{align*}
\Av(\pi)&=\Av(\sigma_1\oplus\sigma_2\oplus\dotsb\oplus\sigma_m)\\
&\subseteq\Av(\sigma_1\oplus1\oplus\sigma_2\oplus1\oplus\dotsb\oplus1\oplus
\sigma_m)\\
&\in 
\FSplit(\Av(\sigma_1\oplus1), \Av(1\oplus\sigma_2), \Av(\sigma_2\oplus1), 
\dotsc,
\Av(1\oplus\sigma_m)).
\end{align*}

Choose any $\sigma\in\{\sigma_1,\dotsc,\sigma_m\}$. By induction, we have
\[
\Av(\sigma)\in\FSplit(\Av(\tau_1),\dotsc,\Av(\tau_\ell)), 
\]
where $\tau_1,\dotsc, \tau_\ell$ are the slim skew-decomposable subpermutations  of~$\sigma$. By Lemma~\ref{lem-1p}, we then get
\begin{align*} 
\Av(1\oplus\sigma)&\in\FSplit(\Av(1\oplus\tau_1),\dotsc,\Av(1\oplus\tau_\ell)),
\quad\text{and}\\
\Av(\sigma\oplus1)&\in\FSplit(\Av(\tau_1\oplus1),\dotsc,\Av(\tau_\ell\oplus1)).
\end{align*}
Note that the permutations $1\oplus\tau_i$ are slim subpermutations of $\pi$,  except perhaps when $\sigma=\sigma_1$, in which case we do not need to split $\Av(1\oplus\sigma)$. An analogous claim holds for $\tau_i\oplus 1$. This completes the proof.
\end{proof}

Of course an analogous statement to Lemma \ref{lem-slim} holds for skew-decomposables as well, by symmetry.

\section{Representable subclasses of \Sep}

Our goal is to describe all the unsplittable subclasses of $\Sep$. We will show that they are precisely the classes obtained by iterated inflations of $\I$, $\D$, $\bac$, $\cab$, $\acb$, and $\bca$ (recall that $\bac$ denotes the class $\Av(213)$, with $\cab, \acb$ and $\bca$ its corresponding symmetries).

We will use $\X\Y$ as a shorthand for $\X[\Y]$. Note that $\X[\uni]=\uni[\X]=\X$ for any class~$\X$. Note also that $\X(\Y\Z)=(\X\Y)\Z$ for any three classes $\X$, $\Y$ and $\Z$. Thus, permutation classes form a monoid with respect to the inflation operation, and $\uni$ is its unit element. By Lemma \ref{lem-infl} the unsplittable proper subclasses of $\Sep$ form a submonoid of this monoid, which will be denoted by~$\cUs$. 

Consider now the alphabet $\cA=\{\I, \D, \bac, \cab, \acb, \bca\}$. Let $\cA^*$ be the set of words over~$\cA$. Each word $w\in\cA^*$ can be seen as a description of an iterated inflation of a sequence of permutation classes, and in particular, each word represents a permutation class from~$\cUs$. Thus, for example, the word $\I\D$ represents the class $\La$ of layered permutations. We take the empty word $\lambda\in\cA^*$ to represent the class~\uni. The subclasses of \Sep that can be represented in this way will be called \emph{representable}, and again by Lemma \ref{lem-infl} form a submonoid of \cUs which we will denote by \Rep. Before proceeding to the main result in the next section (that $\cUs = \Rep$) we will investigate the structure of, and relations between, representable classes more closely.

The interpretation of a word $w \in \cA^*$ as an element of \Rep defines a monoid homomorphism from $\cA^*$ to \Rep. We will frequently blur the distinction between $\cA^*$ and \Rep as the appropriate interpretation will generally be clear by context. When the distinction needs to be made explicit, we shall write $\cl(w)$ for the class corresponding to $w \in \cA^*$. The first thing to note is that the map $\cl \colon \cA^* \to \Rep$ is not an isomorphism. Specifically, using Observations \ref{obs-IandDclosure} and \ref{obs-closure-basis} (and symmetry):

\begin{observation}\label{obs-red} The following identities hold in the
monoid $\Rep$:
\begin{align*}
 \I\I&=\I&\I\cab&=\cab & \cab\D&=\cab\\
 \D\D&=\D&\I\bca&=\bca& \bca\D&=\bca\\
 &&\D\acb&=\acb&\acb\I&=\acb\\
&& \D\bac&=\bac & \bac\I&=\bac
\end{align*}
\end{observation}

Let $\rel$ denote the congruence on $\cA^*$ generated by these identities. Specifically, \rel is the transitive closure of the relation defined between two words in $\cA^*$ if the second can be obtained from the first by replacing either two consecutive symbols from the left hand side of one of the identities by the one on the right, or vice versa. Considering only the process of replacing pairs of symbols by single ones we can interpret the identities as a family of rewrite rules (such as $\I \bca \to \bca$) whose iterated application will allow us to produce a \emph{reduced word} (one which allows no further applications of such rules) in each \rel-equivalence class. In fact it is easily seen that there is a unique reduced word in each such class since the rewrite rules are \emph{locally confluent}. That is, if we begin with a word $w \in \cA^*$ and apply two different rewrite rules to obtain $w'$ and $w''$, then either $w' = w''$ or there is a rewrite rule that applies to $w'$ and another one that applies to $w''$, both leading to the same word $w'''$. This is obvious if the rules that produced $w'$ and $w''$ share no common character in $w$ (since then in some sense the other application is still available). If they do share a common character then sometimes they produce the same word, as for instance $x \cab \D \acb y \to x \cab \acb y$ either by rewriting $\cab \D \to \cab$ or by rewriting $\D \acb \to \acb$, or the other rule is still available as in 
\[
x \, \I \bca \D \, y \to x \, \bca \D \, y \to x \, \bca  \,y \quad  \mbox{or} \quad  x \, \I \bca \D \, y \to x \, \I \bca \,y \to x \, \bca \, y.
\]

We wish to show that two distinct reduced words from $\cA^*$ represent distinct permutation classes. In fact, we prove a more general result, showing that set inclusion among permutation classes corresponds to a natural order relation on words.

\begin{observation}\label{obs-order} The following inclusions hold among
permutation classes:
\begin{align*}
\I\D&\subseteq \bca & \D\I&\subseteq \acb\\
\I\D&\subseteq \cab & \D\I&\subseteq \bac
\end{align*} 
\end{observation}

Let us say that a word $x\in\cA^*$ is \emph{dominated} by a word $y\in\cA^*$ if $x$ can be obtained from $y$ by a (possibly empty) sequence of steps where in each step we either
\begin{itemize}
\setlength{\itemsep}{0pt}
\item erase an arbitrary letter,
\item replace an occurrence of the letter $\bca$ or $\cab$ by two consecutive
letters $\I\D$, or
\item replace an occurrence of the letter $\acb$ or $\bac$ by two consecutive
letters $\D\I$.
\end{itemize}

Domination is a partial order on $\cA^*$. This is clear since it is transitive by definition, and it contains no cycles since each single step either reduces the number of characters from the set $\{\acb, \bac, \bca, \cab\}$ in a word, or leaves that the same but reduces the total length of the word. We will show that the restriction of domination to reduced words corresponds to the inclusion order of representable permutation classes. Before we prove this, we will need several auxiliary claims dealing with the structure of inflations of permutation classes.

\begin{lemma}
\label{lem-l0} 
Let $\X = \cl(x)$ be a permutation class represented by a nonempty reduced word $x$, and let $x_1\in\cA$ be the first symbol of~$x$. Then the class $\X$ is sum-closed if and only if $x_1$ is one of $\I$, $\cab$ or $\bca$, i.e., $x_1$ is itself sum-closed. Symmetrically, $\X$ is skew-closed iff $x_1$ is one of $\D$, $\acb$ or~$\bac$.
\end{lemma}

\begin{proof}
If $x_1$ is equal to $\I$, $\bca$ or $\cab$ then $\I x \rel x$ so $\cl(\I x) = \cl(x) = \X$, i.e., $\X$ is sum closed. Correspondingly if $x_1$ is equal to $\D$, $\acb$ or $\bac$ then $\X$ is skew-closed. The converse follows in each case because no proper subclass of \Sep other than $\uni$ can be both sum- and skew-closed.
\end{proof}

\begin{lemma}
\label{lem-notI} 
Let $\X = \cl(x)$ be a permutation class represented by a nonempty reduced word $x$, and let $x_1\in\cA$ be the first symbol of~$x$. If $x_1 \neq \I$ then for every permutation $\pi\in \X$ there is a sum-indecomposable permutation $\sigma \in \X$ containing $\pi$.
\end{lemma}

\begin{proof}
Let $x = x_1 x^{-}$ and suppose first that $x_1=\cab$, and $\pi \in \X$. Write $\pi = \sigma[\rho_1, \rho_2, \dots, \rho_k]$ for some $\sigma \in \cab$ and $\rho_1, \rho_2, \dots, \rho_k \in \cl(x^{-})$. Then $\pi \ominus 1 = (\sigma \ominus 1) [\rho_1, \rho_2, \dots, \rho_k, 1]$ is in $\X$ since $\sigma \ominus 1 \in \cab$. If $x_1$ is neither $\I$ nor $\cab$, a similar argument shows that for every $\pi \in \X$, we also have $1 \ominus \pi \in \X$.
\end{proof}

\begin{lemma}
\label{lem-lA} 
Let $\X$ be a permutation class represented by a nonempty reduced word $x= \I x^{-}$.  Let $\X^{-} \! = \cl(x^{-})$. A permutation $\pi$ belongs to $\X^{-}$ if and only if $\pi$ is contained in a sum-indecomposable permutation $\sigma$ belonging to~$\X$.
\end{lemma}

\begin{proof}
The result is trivial if $x^{-}$ is the empty word. Otherwise, since $x$ is reduced we know that the first character of $x^{-}$ is not equal to $\I$. 

Suppose that $\pi$ belongs to $\X^-$. Then, applying Lemma~\ref{lem-notI} to the class $\X^-$, we conclude that $\pi$ is contained in a sum-indecomposable permutation $\sigma$ belonging to $\X^-$ (and hence also to $\X$). 

Conversely, suppose $\pi$ is contained in a sum-indecomposable permutation $\sigma \in \X$. Since $\X$ is the sum closure of $\X^{-}$ the sum-indecomposable permutations of $\X$ all belong to $\X^{-}$, so $\sigma \in \X^{-}$ and hence $\pi \in \X^{-}$.
\end{proof}

\begin{lemma}
\label{lem-lB} 
Let $\X$ be a permutation class represented by a nonempty reduced word $x= \bac x^{-}$. Let $\X^{-} \! = \cl(x^{-})$. A permutation $\pi$ belongs to $\X^{-}$ if and only if $\pi$ is contained in a sum-indecomposable
permutation $\sigma$ such that $\sigma \oplus 1 \in \X$.
\end{lemma}

\begin{proof} 
The result is trivial if $x^{-}$ is the empty word. Otherwise, since $x$ is reduced, the first character of $x^{-}$ is not equal to~$\I$. 

Suppose that $\pi$ belongs to $\X^-$. By Lemma~\ref{lem-notI}, there is a sum-indecomposable permutation $\sigma \in \X^{-}$ containing $\pi$, and then clearly $\sigma\oplus 1 = 12[\sigma, 1]$ belongs to~$\X$.
 
Conversely, suppose that $\pi$ is contained in a sum-indecomposable permutation $\sigma$ such that $\sigma\oplus 1$ belongs to~$\X$. We claim that $\sigma$, and hence also $\pi$, belongs to~$\X^-$. Since $\sigma\oplus 1$ is in $\X=\bac \X^{-}$, it can be obtained by inflating a $213$-avoiding permutation
$\rho$ by elements of $\X^{-}$. As $\sigma \oplus 1$ ends with its maximum element, so must $\rho$, and as $\rho \in \bac$, $\rho$ must be increasing. But $\sigma$ is sum-indecomposable, so the only possibility is that $\rho = 12$ and $\sigma \oplus 1 = 12[\sigma, 1]$. Hence $\sigma \in \X^{-}$ as claimed.
\end{proof}

\begin{theorem}\label{thm-order} Let $x,y \in \cA^*$ be two reduced words. Then
$x$ is dominated by $y$ if and only if $\cl(x)$ is a subclass of~$\cl(y)$.
\end{theorem}
\begin{proof}
Clearly, if $x$ is dominated by $y$, then $\cl(x)$ is a subclass of~$\cl(y)$. 

Let us prove the converse. Let $\X=\cl(x)$ and $\Y=\cl(y)$ be the
permutation classes represented by $x$ and $y$, and suppose that $\X\subseteq
\Y$. Our goal is to show that $x$ is dominated by~$y$.

Let $x_1 x_2 \dotsb x_n$ be the sequence of symbols of $x$, and $y_1 y_2 \dotsb
y_m$ the sequence of symbols of~$y$.  We proceed by induction on $m+n$. Note
that if $x$ or $y$ is the empty word, the proof is trivial, so let us
assume that both $x$ and $y$ are nonempty.

Let $x^-$ be the word $x_2 \dotsb x_n$, let $y^-$ be the word $y_2 \dotsb y_m$,
let $\X^{-}$ be the class $\cl(x^-)$ and $\Y^{-}$ the class $\cl(y^-)$. Note that
we clearly have $\X^{-}\subseteq \X$ and $\Y^{-}\subseteq \Y$.

We now distinguish several cases, based on the symbols $x_1$ and $y_1$. We will
restrict ourselves to the situations when $y_1$ is equal to $\I$ or to $\bac$,
as the remaining cases are symmetric to these two.

\textbf{Case 1: $x_1=\I$ and $y_1=\I$.} 
By Lemma~\ref{lem-lA}, $\pi \in \X^-$ if and only if $\pi$ is contained in some sum-indecomposable element of $\X$. But then since $\X \subseteq \Y$, $\pi$ is contained in a sum indecomposable element of $\Y$ and hence, by the same lemma, $\pi$ belongs to $\Y^-$. So $\X^{-} \subseteq \Y^{-}$. Therefore, by induction, $x^{-}$ is dominated by $y^{-}$ and hence $x$ is dominated by $y$.

\textbf{Case 2: $x_1\neq\I$ and $y_1=\I$.} 
In this case, we claim that $\X \subseteq \Y^-$. To see this, choose $\pi \in \X$.  By Lemma~\ref{lem-notI}, there is a sum-indecomposable permutation $\sigma \in \X$ containing~$\pi$. Consequently, $\sigma$ belongs to $\Y$ as well, and by Lemma~\ref{lem-lA}, $\pi$ belongs to~$\Y^-$. By induction, $y^-$ dominates $x$ and therefore $y$ dominates $x$ as well.

\textbf{Case 3: $x_1=\bac$ and $y_1=\bac$.} Applying Lemma~\ref{lem-lB}
to $\X$ and $\Y$, we conclude that $\X^-\subseteq \Y^-$. Hence $x^-$ is dominated by
$y^-$, and $x$ is dominated by~$y$.
 
\textbf{Case 4: $x_1=\D$, $x_2\neq \I$ and $y_1=\bac$.} 
If $x$ has length $1$, then clearly $x$ is dominated by $y$, so assume that $x^-$ is nonempty. Note that since $x$ is reduced, we have either $x_2=\bca$ or $x_2=\cab$, and in particular $\X^-$ is sum-closed. We wish to show that $\X^-\subseteq \Y^-$. Let $\pi\in \X^-$. By Lemma~\ref{lem-notI}, there is a sum-indecomposable permutation $\sigma \in \X^-$ that contains~$\pi$. Since $\X^-$ is sum-closed, $\sigma\oplus 1$ belongs to $\X^-$ and therefore also to $\Y$, and so by Lemma~\ref{lem-lB}, $\pi$ belongs to~$\Y^-$. By induction, $x^-$ is dominated by $y^-$, hence $x$ is dominated by~$y$.

\textbf{Case 5: $x_1=\D$, $x_2=\I$ and $y_1=\bac$.} 
If $x$ has length two we are done, so assume this is not the case. Consider the word $x^{--}=x_3 \dotsb x_n$ and the class $\X^{--}=\cl(x^{--})$. Since $x$ is reduced, $\X^{--}$ is skew-closed. Let us prove that $\X^{--}\subseteq \Y^-$. Let $\pi\in \X^{--}$. The permutation $\sigma=\pi\ominus 1$ is in $\X^{--}$, since $\X^{--}$ is skew-closed. Therefore, $\sigma\oplus 1 \in \X^-$, and hence  $\sigma\oplus 1  \in \Y$. By Lemma~\ref{lem-lB}, $\X^{--} \subseteq \Y^{-}$, hence $x^{--}$ is dominated by $y^{-}$, and $x$ is dominated by~$y$.

\textbf{Case 6: $x_1=\I$ and $y_1=\bac$.} 
In this case, we may consider the class $\X^{+}=\D \X$. Since $\Y$ is skew-closed, $\X^{+}$ is a subclass of $\Y$, and by the argument from the previous case, we obtain that $x^{-}$ is dominated by $y^{-}$, and $x$ by~$y$ (the inductive hypothesis still applies since in the previous case it was applied to $x^{--}$ which is two characters shorter than $x$).

\textbf{Case 7: $x_1\in\{\bca, \cab, \acb\}$ and $y_1=\bac$.} 
We show that $\X \subseteq \Y^{-}$. Let $\pi \in \X$. By Lemma~\ref{lem-notI}, there is a sum-indecomposable permutation $\sigma \in \X$ containing~$\pi$. Moreover, our assumptions on $x_1$ guarantee that $\sigma \oplus 1$ is also in $\X$, and therefore in~$\Y$. By Lemma \ref{lem-lB} $\pi \in \Y^{-}$, hence $x$ is dominated by $y^-$, and therefore also by~$y$.

This completes the proof of the theorem.
\end{proof}

\begin{corollary}
 If $\X$ and $\Y$ are permutation classes represented by reduced words $x$ and
$y$, respectively, then $\X=\Y$ if and only if $x=y$.
\end{corollary}
\begin{proof}
 If $\X=\Y$, then by Theorem~\ref{thm-order} $x$ is dominated by $y$ and $y$ is
dominated by $x$, hence $x=y$. The other direction is trivial.
\end{proof}

\subsection{Bases of representable classes}

We can determine the basis of a representable class easily from the reduced word that represents it. Though we do not require this result for the remainder of the paper, we record the procedure for doing so here.

For a set $F$ of permutations (generally always a subset of \Sep) let $\Avs(F)$ be the set of $F$-avoiding separable permutations, and for a subclass $\X$ of separable permutations, let $\Obs(\X)$ denote the minimal separable permutations not belonging to~$\X$.  With this notation, we always have $\X=\Avs(\Obs(\X))$. We call the elements of $\Obs(\X)$ the \emph{minimal separable obstructions} for~$\X$.

We will prove several general lemmas that will allow us to determine the minimal separable obstructions of any representable permutation class. For a set of permutations $F$, let $F\oplus 1$ denote the set $\{\pi\oplus 1;\; \pi\in F\}$, with $1\oplus F$, $1\ominus F$ and $F\ominus 1$ defined analogously. 

Note that each of the following results has a number of symmetric variations that we do not explicitly specify.

\begin{lemma}\label{lem-forb-bac}
Let $F$ be a set of skew-decomposable permutations, and let $\X=\Avs(F)$. Then
$\bac \X=\Avs(F\oplus 1)$.
\end{lemma}
\begin{proof} Since the elements of $F$ are skew-decomposable, the class $\X$ is sum-closed. Notice that no permutation in $F\oplus 1$ can belong to~$\bac \X$, since then $F$ would belong to $\X$ by Lemma~\ref{lem-lB}.  Therefore $\bac \X\subseteq \Avs(F\oplus 1)$. To prove the opposite inclusion, let $\pi\in\Avs(F\oplus 1)$ and proceed by induction on $|\pi|$ to show that $\pi\in\bac X$. For $\pi=1$ this is clear. Suppose that $\pi$ is a skew sum of the form $\pi=\pi_1\ominus\pi_2$. Then by induction, both $\pi_1$ and $\pi_2$ are in $\bac \X$, and since $\bac \X$ is skew-closed, we are done. Finally, suppose that $\pi=\pi_1\oplus\pi_2$. We then see that $\pi_1$ is in $\Avs(F)=\X$, and $\pi_2$ is an inflation of a permutation $\rho\in\bac$ by elements of $\X$ by induction. It follows that $\pi$ is an inflation of $1\oplus\rho \in \bac$ by elements of $\X$, as claimed.
\end{proof}

\begin{corollary}\label{cor-forb-bac}
If $\X$ is a sum-closed permutation class, then $\Obs(\bac \X)=\Obs(\X)\oplus
1$. If $\X$ is an arbitrary permutation class (not necessarily sum-closed), then
$\Obs(\bac \X)=\Obs(\bac \I \X)=\Obs(\I \X)\oplus 1$.
\end{corollary}

\begin{lemma}\label{lem-forb-i}
 Let $F$ be a set of skew-decomposable permutations, and let $\X=\Avs(F)$. Then $\D \X=\Avs\left((F\oplus 1)\cup (1\oplus F)\right)$.
\end{lemma}
\begin{proof}
 We again easily see that the elements of $\D \X$ must avoid all the
permutations in $1\oplus F$ and $F\oplus 1$. 

Let $\pi\in\Avs((F\oplus 1)\cup (1\oplus F))$, and write $\pi$ as a skew
sum $\pi=\pi_1\ominus\pi_2\ominus\dotsb\ominus\pi_k$ for $k\ge 1$, where each
$\pi_i$ is either equal to 1 or sum-decomposable. We claim that each $\pi_i$
is in $\X$: indeed, if $\pi_i$ is sum-decomposable, each of its summands
must avoid $F$, and since $\X$ is sum-closed, we get $\pi_i\in \X=\Avs(F)$. Thus,
we obtain $\pi\in\D \X$ as claimed.
\end{proof}

\begin{corollary}\label{cor-forb-i}
 If $X$ is a sum-closed permutation class, then $\Obs(\D
\X)=\left(\Obs(\X)\oplus 1\right)\cup\left(1\oplus\Obs(\X)\right)$.
\end{corollary}

The results above are sufficient to characterise the minimal obstructions of a representable class, simply by working through the reduced word for a class from right to left (and using the known obstructions for the basic case). For example, consider the class $\C = \bac \D \bca$. 
\begin{align*}
\Obs(\bca) &= \{231\}  & \\
\Obs(\D \bca) &= \{ 1342, 2314 \} & \mbox{(Lemma \ref{lem-forb-i})} &\\
\Obs(\bac \D \bca) &= \Obs( \I \D \bca) \oplus 1& \mbox{(Corollary \ref{cor-forb-bac})} \\
\Obs(\I \D \bca) &= (\Obs(\D \bca) \ominus 1) \cup (1 \ominus \Obs(\D \bca))  & \mbox{(Symmetry of Lemma \ref{lem-forb-i})}\\
 & =  \{ 51342, 12453, 52314, 13425 \} \\
 \Obs(\bac \D \bca) &= \{ 513426, 124536, 523146, 134256 \}.
 \end{align*}

\begin{corollary}
Let $\X$ be a representable class. Then the elements of $\Obs(\X)$ are slim
separable permutations, all of the same size.
\end{corollary}

\begin{corollary}\label{cor-slimrepre}
For every slim permutation $\pi$, the class $\Avs(\pi)$ is representable. Moreover, if $\pi$ has size at least 3, then $\Avs(\pi)$ has a representation whose every symbol is one of $\{\bac,\acb,\cab,\bca\}$. Conversely, for every representable class, $\C = \cl(c)$, defined by a word, $c$, over the alphabet $\{\bac,\acb,\cab,\bca\}$, $\Obs(\C)$ is a single slim permutation whose length is two more than the number of characters in $c$.
\end{corollary}

\section{Unsplittable subclasses of the separable permutations}

By Proposition 2.7 of \cite{Jelinek2015Splittings-and-} and its symmetries, the classes \acb, \bac, \bca, and \cab are all unsplittable. Combined with Lemma \ref{lem-infl} this implies that all the representable classes are unsplittable. The aim of this section is to prove the converse - that any non-representable class is splittable. The only proper subclasses of \D and \I are finite, and hence splittable, so the first non-trivial result of this type is considered in the next lemma. Recall that $\La = \I \D$ is the class of layered permutations.

%

\begin{lemma}
\label{lem-sub-bca-splittable}
Every proper subclass of $\bca$ is $\La$-splittable.
\end{lemma}

\begin{proof}
Suppose to the contrary that there exist proper subclasses of $\bca$ which are not $\La$-splittable. It is well known (and follows easily by an application of Higman's Lemma~\cite{Higman1952Ordering-by-div}) that the partial ordering of $\Sep$ under permutation containment is a well quasiorder and, as a consequence the containment relation between subclasses of $\Sep$ is well-founded. So, among the non $\La$-splittable proper subclasses of $\bca$ we may choose one, $\C$, minimal with respect to containment.

Choose a minimal permutation $\sigma \in \bca \setminus \C$ (i.e., a basis element of $\C$ relative to $\bca$). If $\sigma$ is not slim, then $\C$ is splittable, and since all proper subclasses of $\C$ are $\La$-splittable, so is $\C$, a contradiction. 

Suppose that $\sigma = \rho \oplus 1$, and let $\pi \in \C$. The elements of $\pi$ less than its final element avoid $\rho$, while those greater than or equal to its final element are decreasing (since $\pi \in \bca$). Hence $\pi \in \D \merge \Av(\rho)$ so $\C$ is $\{ \D, \Av(\rho) \}$-splittable, and once again we have a contradiction. Similarly, if $\sigma = 1 \oplus \rho$ then the elements up to and including $\pi$'s minimum are decreasing, and those after avoid $\rho$, so again we have $\{ \D, \Av(\rho) \}$-splittability and a contradiction.

If $\sigma = \rho \ominus 1$ then, because $\sigma \in \bca$ and $\sigma$ is slim, $\sigma = 21$ and $\C \subseteq \I$, clearly a contradiction.

Finally suppose that $\sigma = 1 \ominus \rho$, where $\rho = 1 \oplus \tau$ or $\rho = \tau \oplus 1$. Let $\pi \in \C$ be sum-indecomposable, i.e., since $\pi \in \bca$, $\pi$ begins with its maximum element. The remainder of $\pi$ avoids $\rho$ so, by the previous argument, it belongs to $\D \merge \Av(\tau)$. Therefore, $\pi \in \D \merge \Av(\tau)$ as well. An arbitrary permutation in $\C$ is the sum of its sum-indecomposable parts and is therefore in $\La \merge \I \Av(\tau)$. But $\C \not \subseteq \I \Av(\tau)$ since $1 \ominus \tau \in \C$. Therefore, by assumption $\C \cap \I \Av(\tau)$ is $\La$-splittable, and hence so is $\C$, our final contradiction.
\end{proof}

The plan from here is to continue essentially inductively. To proceed  we need another technical lemma.

\begin{lemma}\label{lem-inters}
Let $\X$ and $\Y$ be permutation classes. Then $(\D \X)\cap(\D \Y)=\D(\X \cap \Y)$
and $(\bac \X)\cap (\bac \Y)=\bac (\X \cap \Y)$.
\end{lemma}

\begin{proof}
Clearly, if $\pi$ is a permutation from $\D( \X \cap \Y)$, then $\pi$ belongs to
both $\D \X$ and $\D \Y$. To prove the converse, let $\pi \in 
(\D \X)\cap(\D \Y)$. Write $\pi$
as $\pi=\pi_1 \ominus \pi_2 \ominus \dotsb \ominus \pi_k$, where each $\pi_i$
is a skew-indecomposable permutation. Then each $\pi_i$ must belong
to $\X \cap \Y$, and so $\pi$ is in $\D(\X \cap \Y)$, as claimed. This proves the
first identity.

To prove the second one, we again easily observe that $\bac(\X \cap \Y)$ is a
subclass of $(\bac \X) \cap (\bac \Y)$. 

We now proceed by induction on the length of a permutation $\pi \in (\bac \X) \cap (\bac \Y)$ to show that it belongs to $\bac(\X \cap \Y)$. The case $\pi=1$ is trivial. Moreover, if $\pi$ is neither sum-decomposable nor skew-decomposable, then it belongs to $\X \cap \Y$ and we are done. Suppose that
$\pi=\pi_1\ominus\pi_2$, for some $\pi_1$ and $\pi_2$. By induction, both
$\pi_1$ and $\pi_2$ belong to $\bac(\X \cap \Y)$, and so does $\pi$ itself, since
$\bac(\X \cap \Y)$ is skew-closed. 

Finally, suppose that $\pi$ has the form $\pi_1 \oplus \pi_2$, with $\pi_1$ being sum-indecomposable. Since $\pi$ is in $\bac \X$, it was obtained by inflating a permutation $\rho\in\bac$ by elements of~$\X$. Note that the permutation $\pi_1$ was obtained by inflating a single element of $\rho$, since otherwise $\rho$
would contain the pattern~$213$. In particular, $\pi_1$ belongs to $\X$, and by
an analogous argument, it belongs to $\Y$ as well. By induction, $\pi_2$ is in
$\bac(\X \cap \Y)$, and in particular, it can be obtained by inflating a
permutation $\tau \in \bac$ by elements of $\X \cap \Y$. Then $\pi$ itself can be
obtained by inflating the permutation $1\oplus \tau \in \bac$ by elements of
$\X \cap \Y$, and so $\pi$ belongs to~$\bac(\X \cap \Y)$, as claimed.
\end{proof}

The essential part of the remainder of the proof is to show that any proper subclass of a
representable class $\Y$ has a splitting into representable proper subclasses
of~$\Y$. The key step of the proof is presented below as a lemma.

\begin{lemma}
\label{lem-repre} 
Let $\Y$ be a representable class represented by the reduced word $y=y_1 y_2\dotsb y_n$, and let $\sigma\in \Y$ be a slim permutation. Let $\X$ be the class $\Y \cap \Avs(\sigma)$. Then there exist representable classes $\Y_1,\dotsc, \Y_k$ that are proper subclasses of $\Y$, and such that $\X$ is $\{\Y_1,\dotsc,\Y_k\}$-splittable.
\end{lemma}

\begin{proof}
We proceed by induction, firstly over the possible values of $n$, i.e., the length of~$y$, and secondly for a given $n$, over the size of $\sigma$. Note that the case of $n\le 1$ is already handled by Lemma \ref{lem-sub-bca-splittable} and the preceding remark. Also, for any $n$, the case $|\sigma|=2$ is trivial.

Suppose now that $n\ge 2$ and $|\sigma|\ge 3$. Since $\sigma$ is slim, it has one of the four possible forms $1\oplus\rho$, $\rho\oplus 1$, $1\ominus\rho$ or $\rho\ominus 1$ for a slim permutation~$\rho$.

Let $\Y^{-} = \cl(y_2\dotsb y_n)$. By symmetry, we may assume that $y_1$ is either $\D$ or~$\bac$. We will treat the two cases separately.

First, suppose that $y_1=\D$. Since $y$ is reduced, we know that $\Y^{-}$ is sum-closed. By Corollary~\ref{cor-forb-i}, we know that the minimal separable obstructions of $\Y$ are precisely the permutations of the form $1\oplus\tau$ and $\tau\oplus 1$ where $\tau$ is a minimal separable obstruction of~$\Y^{-}$. 

Our first goal will be to give a splitting of $\X$ into classes $\X_1, \X_2,\dotsc, \X_m$ where each $\X_i$ is of one of four types:
\begin{enumerate}[label = (\arabic*)]
\setlength{\itemsep}{0pt}
 \item the class $\uni$,
 \item the class $\Y^{-}$,
 \item the class $\Y \cap \Avs(\rho)$, or
 \item a class of the form $\D \Z$, where $\Z$ is a proper subclass of~$\Y^{-}$.
\end{enumerate}
To obtain such splitting, let a permutation~$\pi\in \X$ be given, and distinguish cases based on the structure of~$\sigma$. 

If $\sigma$ is of the form $1\ominus\rho$, we split $\pi$ into three parts as follows: the first part is the singleton permutation consisting of the leftmost element $\pi_1$ of~$\pi$, the second part consists of all those elements lying above $\pi_1$, and the third part consists of all the elements lying below $\pi_1$.

Observe that the second part forms a permutation from $\Y^{-}$ -- indeed, if the second part contained any minimal separable obstruction $\tau$ of $\Y^{-}$, then $\pi$ would contain $1\oplus\tau$, which is an obstruction of $\Y$, contradicting the assumption $\pi\in \X \subseteq \Y$.

Similarly, the third part of the splitting avoids the permutation $\rho$, since $\pi$ avoids~$\sigma = 1 \ominus \rho$. Thus we can merge any permutation in $\X$ from $\uni$, $\Y^{-}$ and $\Y \cap \Avs(\rho)$, which are classes of types (1), (2) and (3), respectively.

An analogous argument can be made in the case when $\sigma$ has the form $\rho\ominus 1$. Fix the index $i$ such that $\pi_i=1$, and split $\pi$ into three parts, where the first part is the singleton $\pi_i$, the second part is formed by the elements to the left of $\pi_i$, and the third part are the remaining elements. We again see that the second part avoids~$\rho$, and the third part is in~$\Y^-$.

Suppose now that $\sigma$ has the form $1\oplus\rho$ or $\rho\oplus 1$. Then $\sigma$ belongs not just to $\Y$ but actually also to~$\Y^{-}$. Therefore, $\X$ is a subclass of $\D (\Y^{-} \cap \Avs(\sigma))$, and in particular, $\X$ itself is a class of type~(4).

We now need to argue that the classes of types (1) through (4) can be split into representable subclasses of~$\Y$. For types (1) and (2), this is obvious, and for the class of type (3), it follows by induction since $\rho$ is shorter than $\sigma$.

Consider now a class of type (4), that is, a class equal to $\D \Z$ with $\Z$ a proper subclass of~$\Y^{-}$. We may assume, without loss of generality, that $\Z$ is equal to $\Y^-\cap \Avs(\tau)$ for some $\tau\in \Y^{-}$. Since the word $y^-$ is shorter than $y$, we know by the inductive hypothesis that $\Z$ is $\{\Y^{-}_1,\dotsc,\Y^{-}_k\}$-splittable, with $\Y^{-}_1,\dotsc,\Y^{-}_k$ being  representable proper subclasses of~$\Y^{-}$. By Observation~\ref{obs-inf}, $\D \Z$ is $\{\D \Y^{-}_1,\dotsc,\D \Y^{-}_k\}$-splittable. Moreover, since the word $y^{-}$ strictly dominates the reduced representation of any of the classes $\Y^{-}_i$, we see that $y$ strictly dominates the reduced representation of $\D \Y^{-}_i$, and therefore all the classes $\D \Y^-_i$ are proper subclasses of~$\Y$ by Theorem~\ref{thm-order}. We thus conclude that any subclass of $\Y$ of Type 4 has a splitting into representable proper subclasses of~$\Y$.

This completes the induction step for the case when $y_1=\D$. Let us deal with the case of~$y_1=\bac$. The proof of this case follows the same basic structure as the proof of the case $y_1=\D$. By Corollary~\ref{cor-forb-bac}, $\Obs(\Y)=\Obs(\I \Y^{-})\oplus 1$.

This time, we will first show that the class $\X$ is $\{\X_1,
\X_2,\dotsc,\X_m\}$-splittable where each $\X_i$ is of one of these four types:
\begin{enumerate}[label=(\arabic*)]
\setlength{\itemsep}{0pt}
 \item the singleton class $\uni$,
 \item the class $\D\I \Y^{-}$,
 \item the class $\Y \cap \Avs(\rho)$, or
 \item a class of the form $\bac \Z$, where $\Z$ is a proper
subclass of~$\Y^{-}$.
\end{enumerate}
To see this, let $\pi\in \X$, and distinguish cases based on the structure
of $\sigma$.

If $\sigma$ is of the form $1\ominus\rho$, we let $\pi_i$ be the largest element of $\pi$, and split $\pi$ into three parts: the elements to the left of $\pi_i$, the element $\pi_i$ itself, and the elements to the right of~$\pi_i$. The first part then belongs to $\I \Y^{-}$, the second to $\uni$, and the third to $\Y \cap \Avs(\rho)$.

Similarly, if $\sigma$ equals $\rho\ominus 1$, we split $\pi$ into the elements less than the rightmost element, the rightmost element itself, and the elements greater than the rightmost element. The three parts then belong to $\I \Y^{-}$, $\uni$ and $\Y \cap \Avs(\rho)$, respectively.

Suppose that $\sigma$ equals $1\oplus\rho$. Write $\pi$ as $\pi=\pi_1\ominus\pi_2\ominus\dotsb\ominus\pi_\ell$, where all the $\pi_i$ are skew-indecomposable. 
Since $\sigma$ (and hence also $\rho$) is slim, we know that $\rho$ has one of the 
two forms $\rho=1\ominus \eta$ or $\rho=\eta\ominus 1$, for a skew-indecomposable 
permutation~$\eta$. Suppose that $\rho$ equals $\eta\ominus 1$, the other case being 
analogous. We can then partition each permutation $\pi_i$ into three parts: the 
elements smaller than the rightmost one (which belong to $\I \Y^{-}$), the rightmost 
element itself, and the elements larger than the rightmost one (which avoid $\eta$).
Thus, the permutation $\pi$ itself can be merged from three permutations, belonging 
respectively to $\D\I \Y^{-}$, $\D$, and $\Y\cap\Avs(\eta)$, where we use the fact 
that $\D\Avs(\eta)=\Avs(\eta)$, since $\eta$ is skew-indecomposable.

%

Suppose now that $\sigma=\rho\oplus 1$. Then $\Avs(\sigma)=\bac \Avs(\rho)$ by Lemma~\ref{lem-forb-bac}. Therefore, by Lemma~\ref{lem-inters}, we have the identity 
\[
\Y\cap\Avs(\sigma)=\bac \Y^{-} \cap \bac\Avs(\rho)=\bac
(\Y^{-}\cap\Avs(\rho)). 
\]
Let us argue that $\Y^-\cap\Avs(\rho)$ is a proper subclass of $\Y^{-}$, i.e., that $\rho$ is in $\Y^{-}$. If $\rho$ is not in $\Y^{-}$, then $\rho$ cannot be in $\I \Y^{-}$ either, since $\rho$ is skew-decomposable. Thus $\rho$ contains a minimal separable obstruction $\tau$ of $\I \Y^{-}$, and therefore $\sigma$ contains the minimal separable obstruction $\tau\oplus 1$ of $\Y$, which is impossible, since $\sigma$ is in~$\Y$.

What remains is to argue that each class of one of the four types above has a splitting into proper representable subclasses of~$\Y$. This is clear for classes of types (1) and (2), and it follows from the inductive hypothesis for type (3). For type (4), the argument is analogous to the corresponding part of the argument in the case of $y_1=\D$, solved
previously.
\end{proof}

We are finally ready to prove our main result, which describes all the unsplittable proper subclasses of~$\Sep$.

\begin{theorem}\label{thm-main}
A proper subclass of $\Sep$ is unsplittable if and only if it is representable.
\end{theorem}

\begin{proof}
We already know that every representable class is unsplittable. To prove the converse,
let us first observe that every unsplittable proper subclass $\X$ of $\Sep$ is a subclass of a representable class~$\Y$. To see this, choose $\pi$ to be a minimal separable obstruction of~$\X$. We know that $\pi$ is slim, otherwise $\X$ would be splittable by Lemma~\ref{lem-slim}. Thus, $\X$ is a subclass of the
class $\Y=\Avs(\pi)$, which is representable by Corollary~\ref{cor-slimrepre}.

From Theorem~\ref{thm-order} and from the classical Higman's Lemma~\cite{Higman1952Ordering-by-div}, we may deduce that the ordering of the representable classes by inclusion is well-founded. In particular, for any unsplittable class $\X$, there is a (not necessarily unique) minimal representable superclass.

Suppose now that $\X$ is an unsplittable class and assume, for contradiction, that $\X$ is not representable. Let $\Y$ be a minimal representable superclass of~$\X$ and let $\sigma$ be a minimal separable obstruction of $\X$ that belongs to~$\Y$. Note that $\sigma$ is a slim permutation, otherwise $\X$ would be splittable by Lemma~\ref{lem-slim}.

By Lemma~\ref{lem-repre}, the class $\Y\cap\Avs(\sigma)$, and therefore also its subclass $\X$, is $\{\Y_1,\dotsc,\Y_k\}$-splittable, for some representable
proper subclasses $\Y_1,\dotsc,\Y_k$ of~$\Y$. Since $\Y$ was a minimal representable superclass of $\X$, none of the $\Y_i$ contains $\X$. Therefore $\X$ is splittable, contrary to our assumption.
\end{proof}

\begin{theorem}\label{thm-sepsplit}
Let $\X$ be a proper subclass of $\Sep$ and let $A$ be the set of representable subclasses of~$\X$. 
Then $\X$ is $A$-splittable.
\end{theorem}
\begin{proof}
Let $\X$ be a proper subclass of $\Sep$ and let $\Obs(\X)$ be its minimal separable obstructions. 
Note that $\Obs(\X)$ is finite, since $\Sep$ is known to be well-ordered by permutation 
containment (see Atkinson et al.~\cite[Corollary 2.6]{Atkinson2002Partially-well-}). We may assume without loss of generality that all 
the elements of $\Obs(\X)$ are slim, for otherwise we could use Lemma~\ref{lem-slim}
to show that $\X$ can be split into its subclasses whose minimal separable obstructions 
are slim.

Let $\Y$ be a minimal representable superclass of $\X$ (as argued in the proof of Theorem~\ref{thm-main},
such a representable superclass exists). Note that by Theorem~\ref{thm-order},
$\Y$ has only finitely many proper representable subclasses. We will proceed by 
induction on the number of representable subclasses of~$\Y$. If $\Y$ has no representable 
subclass beyond itself, then $\Y=\uni$ and the claim of the theorem is trivial.

Suppose $\Y\neq\uni$. If $\X=\Y$, the claim is again trivial. If $\X$ is a proper 
subclass of $\Y$, Lemma~\ref{lem-repre} shows that $\X$ is $\{\Y_1,\dotsc,\Y_k\}$-splittable,
where each $\Y_i$ is a representable proper subclass of $\Y$, and in particular, $\X$ 
is also $\{\X\cap\Y_1,\dotsc,\X\cap\Y_k\}$-splittable. Each class $\X\cap\Y_i$ is 
contained in the representable class $\Y_i$ which has fewer representable subclasses 
than $\Y$, hence we may apply induction to show that each $\X\cap\Y_i$ is splittable
over the set of its representable subclasses. It follows that $\X$ is splittable 
over its representable subclasses as well.
\end{proof}

As an easy consequence, we obtain the following result.

\begin{theorem}\label{thm-ab}
Let $A$ be a set of representable permutation classes. Let $B$ be the set of minimal 
representable classes that are not contained in any class in~$A$. 
Then $B$ is finite, and a proper subclass $\X$ of $\Sep$ is $A$-splittable
if and only if $\X$ does not contain any element of $B$ as a subclass.
\end{theorem}
\begin{proof}
Since $B$ is an antichain in the inclusion order of representable 
classes, $B$ is finite. Let $\X$ be a proper subclass of $\Sep$. If $\X$ is $A$-splittable, then every representable 
subclass of $\X$ is $A$-splittable as well, and since representable classes are 
unsplittable, every representable subclass of $\X$ is contained in an element of~$A$.
In particular, no subclass of~$\X$ is in~$B$.

Conversely, if $\X$ has no element of $B$ as subclass, then all the representable subclasses
of $\X$ are subclasses of elements of $A$, and so $\X$ is $A$-splittable by Theorem~\ref{thm-sepsplit}.
\end{proof}

\section{Concluding remarks}

In the preceding sections we have answered the question: \emph{what characterises the unsplittable subclasses of $\Sep$?} 
This invites the question: \emph{given a collection, $A$, of unsplittable subclasses of $\Sep$, 
can you characterise the $A$-splittable classes?} 
The $A$-splittable subclasses of $\Sep$ are characterised by Theorem~\ref{thm-ab}. 
It is a natural problem to extend such classification to more general classes.

For certain specific cases of $A$, the characterisation in Theorem~\ref{thm-ab} 
remains valid even without the assumption that $\X$ is a subclass of $\Sep$. For 
instance, it can be shown that an arbitrary class $\X$ is $\uni$-splittable if and only if it does not contain 
$\I$ and $\D$ as subclasses, it is $\I$-splittable iff it does not contain $\D$, 
it is $\{\I,\D\}$-splittable iff it does not contain $\I\D$ and $\D\I$ (see Vatter~\cite{Vatter2016An-Erdos--Hajna}), 
it is $\{\bca,\cab\}$-splittable if it does not contain $\D\I$, etc.

Unfortunately, the assumption that $\X$ is a subclass of $\Sep$ in Theorem~\ref{thm-ab} 
cannot in general be omitted. For instance, it can be shown that the class $\X=\Av(1423)$, 
which is not a subclass of $\Sep$, cannot be split over its representable subclasses. 
In fact, it can be shown that this class is not even $\Sep$-splittable.

To extend our results from separable permutations to more general cases, we first need to get a better understanding 
of the unsplittable classes, especially those that cannot be obtained 
by inflations of smaller unsplittable classes. Our results show that the inflation 
monoid is closely related to splittability. Therefore, it seems natural
to investigate splittability in inflation-closed classes. This motivates the following question.

\begin{question}
What are the unsplittable subclasses of the least inflation-closed class containing $2413$ (or both $2413$ and $3142$)?
\end{question}

We believe that a better understanding of the consequences of splittability in permutation classes, and a more general understanding of the characteristics of unsplittable permutation classes could be a key element in the program of systematically investigating permutation classes by structural means.

\subsection{Acknowledgement}

We thank Vincent Vatter who provided the tikz code for Figures \ref{fig-sums} and \ref{fig-479832156}. The second author gratefully acknowledges the hospitality of the Computer Science Department of the University of Otago, where most of the research presented in this paper was conducted. 

\bibliographystyle{plain}
\bibliography{split.bib}

\begin{bibdiv}
\begin{biblist}

\bib{Albert2007On-the-length-o}{article}{
      author={Albert, M.~H.},
       title={On the length of the longest subsequence avoiding an arbitrary
  pattern in a random permutation},
        date={2007},
        ISSN={1042-9832},
     journal={Random Structures Algorithms},
      volume={31},
      number={2},
       pages={227\ndash 238},
         url={http://dx.doi.org/10.1002/rsa.20140},
      review={\MR{2343720}},
}

\bib{Albert2010Growth-rates-fo}{article}{
      author={Albert, M.~H.},
      author={Atkinson, M.~D.},
      author={Brignall, R.},
      author={Ru{\v{s}}kuc, N.},
      author={Smith, Rebecca},
      author={West, J.},
       title={Growth rates for subclasses of {${\rm Av}(321)$}},
        date={2010},
        ISSN={1077-8926},
     journal={Electron. J. Combin.},
      volume={17},
      number={1},
       pages={Research Paper 141, 16},
         url={http://www.combinatorics.org/Volume_17/Abstracts/v17i1r141.html},
      review={\MR{2745694}},
}

\bib{Atkinson2002Partially-well-}{article}{
      author={Atkinson, M.~D.},
      author={Murphy, M.~M.},
      author={Ru{\v{s}}kuc, N.},
       title={Partially well-ordered closed sets of permutations},
        date={2002},
        ISSN={0167-8094},
     journal={Order},
      volume={19},
      number={2},
       pages={101\ndash 113},
         url={http://dx.doi.org/10.1023/A:1016500300436},
      review={\MR{1922032}},
}

\bib{Bona2012On-the-Best-Upp}{article}{
      author={{B\'{o}na}, M.},
       title={{On the Best Upper Bound for Permutations Avoiding A Pattern of a
  Given Length}},
        date={2012-09},
     journal={ArXiv \href{http://arxiv.org/abs/1209.2404}{1209.2404}},
      eprint={1209.2404},
}

\bib{Bona2014A-New-Upper-Bou}{article}{
      author={B{\'{o}}na, M.},
       title={A new upper bound for 1324-avoiding permutations},
        date={2014},
     journal={Combinatorics, Probability {\&} Computing},
      volume={23},
      number={5},
       pages={717\ndash 724},
         url={http://dx.doi.org/10.1017/S0963548314000091},
}

\bib{BBL}{article}{
      author={Bose, P.},
      author={Buss, J.~F.},
      author={Lubiw, A.},
       title={Pattern matching for permutations},
        date={1998},
        ISSN={0020-0190},
     journal={Information Processing Letters},
      volume={65},
      number={5},
       pages={277 \ndash  283},
  url={http://www.sciencedirect.com/science/article/pii/S0020019097002093},
}

\bib{Brignall2010A-survey-of-sim}{incollection}{
      author={Brignall, R.},
       title={A survey of simple permutations},
        date={2010},
   booktitle={Permutation patterns},
      series={London Math. Soc. Lecture Note Ser.},
      volume={376},
   publisher={Cambridge Univ. Press, Cambridge},
       pages={41\ndash 65},
         url={http://dx.doi.org/10.1017/CBO9780511902499.003},
      review={\MR{2732823}},
}

\bib{Claesson2012Upper-bounds-fo}{article}{
      author={Claesson, A.},
      author={Jel{\'{\i}}nek, V.},
      author={Steingr{\'{\i}}msson, E.},
       title={Upper bounds for the {S}tanley-{W}ilf limit of 1324 and other
  layered patterns},
        date={2012},
        ISSN={0097-3165},
     journal={J. Combin. Theory Ser. A},
      volume={119},
      number={8},
       pages={1680\ndash 1691},
         url={http://dx.doi.org/10.1016/j.jcta.2012.05.006},
      review={\MR{2946382}},
}

\bib{Fraisse2000Theory-of-relat}{book}{
      author={Fra{\"{\i}}ss{\'e}, R.},
       title={Theory of relations},
     edition={Revised},
      series={Studies in Logic and the Foundations of Mathematics},
   publisher={North-Holland Publishing Co., Amsterdam},
        date={2000},
      volume={145},
        ISBN={0-444-50542-3},
        note={With an appendix by Norbert Sauer},
      review={\MR{1808172}},
}

\bib{Higman1952Ordering-by-div}{article}{
      author={Higman, G.},
       title={Ordering by divisibility in abstract algebras},
        date={1952},
        ISSN={0024-6115},
     journal={Proc. London Math. Soc. (3)},
      volume={2},
       pages={326\ndash 336},
      review={\MR{0049867}},
}

\bib{Jelinek2015Splittings-and-}{article}{
      author={Jel{\'{\i}}nek, V.},
      author={Valtr, P.},
       title={Splittings and {R}amsey properties of permutation classes},
        date={2015},
        ISSN={0196-8858},
     journal={Adv. in Appl. Math.},
      volume={63},
       pages={41\ndash 67},
         url={http://dx.doi.org/10.1016/j.aam.2014.10.003},
      review={\MR{3283829}},
}

\bib{Kezdy1996Partitioning-pe}{article}{
      author={K{\'e}zdy, Andr{\'e}~E.},
      author={Snevily, Hunter~S.},
      author={Wang, Chi},
       title={Partitioning permutations into increasing and decreasing
  subsequences},
        date={1996},
        ISSN={0097-3165},
     journal={J. Combin. Theory Ser. A},
      volume={73},
      number={2},
       pages={353\ndash 359},
  url={http://www.sciencedirect.com/science?_ob=GatewayURL&_origin=MR&_method=citationSearch&_piikey=S009731659690028X&_version=1&md5=6bbcefbc20bb1945ef39adef7ec8eedf},
      review={\MR{1370138}},
}

\bib{Vatter2016An-Erdos--Hajna}{article}{
      author={Vatter, V.},
       title={An {E}rd{\H o}s--{H}ajnal analogue for permutation classes},
        date={2016},
     journal={Discrete Mathematics and Theoretical Computer Science},
      volume={8},
      number={2},
       pages={4},
}

\bib{Vatter2015Permutation-cla}{incollection}{
      author={Vatter, Vincent},
       title={Permutation classes},
        date={2015},
   booktitle={Handbook of enumerative combinatorics},
      series={Discrete Math. Appl. (Boca Raton)},
   publisher={CRC Press, Boca Raton, FL},
       pages={753\ndash 833},
      review={\MR{3409353}},
}

\end{biblist}
\end{bibdiv}

\end{document}